\documentclass[10pt]{amsart}
\usepackage[english]{babel}
\usepackage{amsmath,amsthm}
\usepackage{amsfonts,float,array,enumitem}
\usepackage{graphicx}
\usepackage{color}
\usepackage{longtable}
\usepackage{url}
\usepackage{multicol}
\usepackage{supertabular}
\usepackage{tabu}
\usepackage{booktabs}
\usepackage{microtype,hyperref,psfrag,pdfpages,enumitem,afterpage}
\usepackage[nameinlink,nosort]{cleveref}

\crefname{lem}{Lemma}{Lemmas}
\crefname{section}{Section}{Sections}
\crefname{thm}{Theorem}{Theorems}

\newtheorem{lem}{Lemma}
\newtheorem{prop}[lem]{Proposition}
\newtheorem{thm}[lem]{Theorem}

\theoremstyle{remark}
\newtheorem*{rem}{Remark}
\numberwithin{equation}{section}

\newcommand{\spin}{\ifmmode{\rm Spin}\else{${\rm spin}$\ }\fi}
\newcommand{\spinc}{\ifmmode{{\rm Spin}^c}\else{${\rm spin}^c$}\fi}

\newcommand{\gsm}{g_{4}}
\newcommand{\gtop}{g_{4}^{{\rm top}}}
\newcommand{\galg}{g_{\rm alg}}
\newcommand{\ualg}{u_{\rm alg}}

\makeatletter
\newcommand{\addresseshere}{%
  \enddoc@text\let\enddoc@text\relax
}
\makeatother

\makeatletter
\let\mcnewpage\newpage
\newcommand{\changenewpage}{%
  \renewcommand\newpage{%
    \if@firstcolumn
      \hrule width\linewidth height0pt
      \columnbreak
    \else
      \mcnewpage
    \fi
}}
\makeatother

\newcommand{\myemail}[1]{\href{mailto:#1}{#1}}
\let\cref\Cref
\DeclareMathOperator\disc{disc}

\title{On calculating the slice genera of 11- and 12-crossing knots}%
\author{Lukas Lewark}
\address{University of Bern, Mathematical Institute, Alpeneggstr. 22, 3012 Bern, Switzerland}
\email{\myemail{lukas.lewark@math.unibe.ch}}
\author{Duncan McCoy}%
\address{Department of Mathematics, University of Texas at Austin, Austin, Texas, USA}%
\email{\myemail{d.mccoy@math.utexas.edu}}
\date{}%
\hypersetup{colorlinks={true},linkcolor={black},citecolor={black},filecolor={black},urlcolor={black},pdfauthor={\authors},pdftitle={\shorttitle}}
\renewcommand{\arraystretch}{1.2}
\begin{document}
\begin{abstract}
This paper contains the results of efforts to determine values of the smooth and the topological slice genus of 11- and 12-crossing knots.
Upper bounds for these genera were produced by using a computer to search for genus one concordances between knots. For the topological slice genus further upper bounds were produced using the algebraic genus.
Lower bounds were obtained using a new obstruction from the Seifert form and by use of Donaldson's diagonalization theorem.
These results complete the computation of the topological slice genera for all knots with at most 11 crossings and leaves the smooth genera unknown for only two 11-crossing knots.
For 12 crossings there remain merely 25 knots whose smooth or topological slice genus is unknown.
\end{abstract}
\maketitle
\section{Introduction}\label{sec:intro}
For a knot $K$ in $S^3$ the \emph{smooth slice genus} $\gsm(K)$ is the minimal possible genus of a compact oriented surface smoothly and properly embedded into $B^4$ with boundary~$K$.
A knot $K$ is said to be \emph{smoothly slice} if $\gsm(K)$  is zero.
Replacing the condition `smooth' by `locally flat' yields the so-called \emph{topological slice genus} $\gtop(K)$,
and the notion of \emph{topologically slice} knots.
Since a smoothly embedded surface is necessarily locally flat, we always have the inequality
$\gtop(K) \leq \gsm(K)$.
However, in general, $\gtop(K)$ can be strictly smaller than $\gsm(K)$, as was first seen as a consequence of the work of Freedman and Donaldson on 4-manifolds.
For example, the $(-3,5,7)$-pretzel knot is known to be topologically, but not smoothly, slice \cite{rasmussen10khovanov}.
Knots for which the smooth and topological slice genera differ are of particular interest: they exhibit the divergence of smooth and topological behaviour in four dimensions, allowing, for example, the construction of exotic $\mathbb{R}^4$'s \cite[\S9.4]{GompfStip}.

While the three-dimensional genus of a knot can be calculated using Haken's algorithm or knot Floer homology,
it is a hard problem, for which no algorithm is known, to calculate the smooth or topological slice genus of a given knot.
Over the years, a rich variety of tools that can be used in the calculation of slice genera have been devised. For the topological slice genus, these include Freedman's disk theorem, which asserts that any knot with Alexander polynomial one is topologically slice \cite{Freedman_82_TheTopOfFour-dimensionalManifolds}, as well as lower bounds coming from various knot signatures \cite{Kauffman76signature,Levine1969,tristram_1969}, Casson-Gordon invariants \cite{CassonGordon_86} and $L^2$-invariants \cite{Cha_08_TopMinGenusandL2}. For the smooth slice genus, notable lower bounds include those arising from gauge theory \cite{KronheimerMrowka_Gaugetheoryforemb}, Heegaard Floer homology \cite{MR2026543} and Khovanov homology \cite{rasmussen10khovanov}.

This paper contains the results of efforts to determine unknown\footnote{Unknown = `listed as unknown on KnotInfo \cite{Knotinfo} at the time of writing'.} values of the smooth and topological slice genus for 11- and 12-crossing knots.

\begin{table}[t]
\begin{center}%
\hfill{%
\parbox{4cm}{%
\begin{tabular}{ | l | c | c |}
    \hline
    Knot      & $\gsm(K)$ & $\gtop(K)$  \\ \hline
    $11n34$   & 0 or 1    & 0           \\ \hline
    $11n80$   & 1 or 2    & 1           \\ \hline
    $12a153$  & 1 or 2    & 1           \\ \hline
    $12a187$  & 1 or 2    & 1           \\ \hline
    $12a230$  & 1 or 2    & 1           \\ \hline
    $12a244$  & 2         & 1 or 2      \\ \hline
    $12a317$  & 1 or 2    & 1           \\ \hline
    $12a450$  & 1 or 2    & 1           \\ \hline
    $12a570$  & 1 or 2    & 1           \\ \hline
    $12a624$  & 1 or 2    & 1           \\ \hline
    $12a636$  & 1 or 2    & 1           \\ \hline
    $12a810$  & 2         & 1 or 2      \\ \hline
    $12a905$  & 1 or 2    & 1 or 2      \\ \hline
    $12a1142$ & 2         & 1 or 2      \\ \hline
\end{tabular}}\hfill
\parbox{4cm}{%
\begin{tabular}{ | l | c | c |}
    \hline
    Knot      & $\gsm(K)$ & $\gtop(K)$  \\ \hline
    $12a1189$ & 1 or 2    & 1           \\ \hline
    $12a1208$ & 1 or 2    & 1           \\ \hline
    $12n52$   & 1 or 2    & 1           \\ \hline
    $12n63$   & 1 or 2    & 1           \\ \hline
    $12n225$  & 1 or 2    & 1           \\ \hline
    $12n239$  & 1 or 2    & 1           \\ \hline
    $12n512$  & 1 or 2    & 1           \\ \hline
    $12n549$  & 2         & 1 or 2      \\ \hline
    $12n555$  & 1 or 2    & 1 or 2      \\ \hline
    $12n558$  & 1 or 2    & 1           \\ \hline
    $12n642$  & 2         & 1 or 2      \\ \hline
    $12n665$  & 1 or 2    & 1           \\ \hline
    $12n886$  & 1 or 2    & 1           \\ \hline
\end{tabular}\\[1.1\baselineskip]}}
\hfill\mbox{}%
\end{center}%
\vspace{\baselineskip}
\caption{The remaining unknown values.}
\label{table:unknown}
\end{table}
In order to determine the slice genus of a knot one needs to produce an upper bound, typically by exhibiting a surface cobounding the knot, and then establishing a lower bound which shows that surface has optimal genus.
We apply a variety of methods to produce new upper and lower bounds.

Our lower bounds come from two sources. The first is an invariant $t$ due to Taylor, which is a lower bound for the topological slice genus. Though $t$ is determined by the Seifert form,
it is generally difficult to compute. In order to apply it, we deduce from $t$ a new, computable obstruction to $\gtop \leq 1$ by applying the theory of quadratic forms
over the $p$-adic numbers to the symmetrization of the Seifert form.
Our other source of lower bounds is Donaldson's diagonalization theorem, from which we obtain an obstruction to
the smooth slice genus being equal to the signature bound, i.e.~$\gsm = |\sigma| / 2$ for alternating knots.

The majority of our upper bounds arise from a computer search to find knots related by crossing changes and crossing resolutions and using that the slice genus (smooth or topological)
of knots related by a concordance of genus one differs by at most one.
In other cases, we obtained upper bounds for $\gtop$ by computing the recently introduced algebraic genus. Similar to Taylor's invariant, the algebraic
genus is an invariant depending only on the Seifert form. That it is an upper bound is a consequence of Freedman's disk theorem.

Altogether these methods allow us to complete the calculation of $\gtop$ for 11-crossing knots.
In total, for knots with up to 12 crossings, there remain 22 unknown values for the smooth slice genus, and 7 unknown values for the topological slice genus.

\Cref{table:unknown} summarizes the prime knots with crossing number at most 12 for which the topological and smooth slice genera are not known and their possible values. We hope that this paper will be helpful in drawing attention to these remaining unknown values, which may well require new and more interesting techniques to determine.

Each of the four following sections is devoted to one of the slice genus bounds we used.
First, we looked for genus one concordances (\cref{sec:genusone}); then, we applied the
more sophisticated tools to the remaining unknown genera: the obstruction from Taylor's bound (\cref{sec:taylor}),
the obstruction from Donaldson's theorem (\cref{sec:Donaldson}), and the upper bound from the algebraic genus
(\cref{sec:galg}).
Details about calculations are contained in the appendices.
\section*{Erratum and  Addendum (August 2023)}
We are grateful to Stepan Orevkov for pointing the following out to us~\cite{orevkov}.
\begin{itemize}
\item In the previous version of this paper, condition (1) in \cref{thm:padic} was erroneously stated as
\smallskip
\begin{quote}
There is a non-negative integer $e$ and an integer $n$
such that $\disc \eta = p^{2e}(np + 1)$, and the Hasse symbol of $\eta$ over $\mathbb{Q}_p$ is $-1$.
\end{quote}
\smallskip
In the current version of this paper, the condition has been corrected to:
\smallskip
\begin{quote}
There is an integer $e\geq 0$, an integer $n$, and an integer $r$ not divisible by $p$,
such that $\disc \eta = p^{2e}(np + r^2)$, and the Hasse symbol of $\eta$ over $\mathbb{Q}_p$ is $-1$.
\end{quote}
\smallskip
Since the correction makes condition (1) weaker, it makes \cref{thm:padic} a stronger tool to show $t(K) \geq 2$ and thus $\gtop(K) \geq 2$ for knots $K$. In particular, our previous computations of $t$ and $\gtop$ for various knots remain unaffected by this correction.\smallskip
\item \cref{thm:padic} may actually be applied to the knot $K =$ 12a1142, since $\disc K = -99 = 3^2\cdot (-4\cdot 3 + 1)$.  This implies $t(K) = \gtop(K) = \galg(K) = 2$.\smallskip
\item The Taylor invariant of $K = $ 12n549 is $1$. Indeed, the Seifert matrix
\[
\begin{pmatrix}
-1 & 0 & 0 & -1 & 0 & 0 & 0 & 0\\
0 & -1 & 0 & 0 & 0 & -1 & -1 & -1\\
0 & 0 & 0 & -1 & 0 & 0 & 0 & 0\\
0 & 0 & 0 & -2 & 0 & 0 & 0 & 0\\
0 & 0 & -1 & -1 & 0 & 0 & 0 & 0\\
0 & 0 & -1 & 0 & -1 & -1 & -1 & 0\\
0 & 0 & -1 & -1 & -1 & 0 & 0 & 0\\
0 & 0 & 0 & 0 & -1 & -1 & -1 & -1
\end{pmatrix}
\]
of $K$ (taken from~\cite{Knotinfo})
admits an isotropic subgroup with basis
\[
(0,0,-2,1,0,0,0,0),\quad
(1,1,1,-1,1,-1,1,-1),\quad
(0,1,0,0,0,-1,2,-1).
\]
\end{itemize}
Moreover, we have noticed the following:
\begin{itemize}
\item The algebraic genus of $K = $ 12a735 is 1. Indeed, the Seifert matrix
\[
\begin{pmatrix}
1 &  0 &  0 &  0\\
 -1 &  -1 &  1 &  -1\\
 1 &  0 &  -2 &  1\\
 -1 &  0 &  0 &  -3
\end{pmatrix}
\]
of $K$ (taken from~\cite{Knotinfo}) admits an Alexander-trivial subgroup with basis
\[(4,3,1,-2),\quad(1,1,0,0).\]
\item At the end of \cref{sec:galg}, we claimed that
\smallskip
\begin{quote}
A sharp upper bound for the algebraic genus is given in 2341 cases by the algebraic unknotting number, i.e.~$\galg = \ualg$;
in the other 582 cases, we explicitly found an Alexander-trivial subgroup of sufficient rank (for those knots, $\galg = \ualg - 1$).
\end{quote}
\smallskip
These numbers should instead be 2324 and 599, respectively.
\item 
The transformation matrices in the ancillary text file distributed with the previous arXiv version were not correct. We have fixed that in the current version.
\end{itemize}

We have not updated the body of this paper, with the exception of correcting \cref{thm:padic}.
However, for the convenience of the reader,
we briefly summarize here the developments since the previous version of this paper (November 2016)
regarding the computation of the knot invariants $t, \gsm, \gtop, \galg$ for knots in the table of prime knots with crossing number~12 or less:
\begin{itemize}
\item Since the two knots 12a1142 and 12n549 mentioned above were the last ones
with unknown $t$, the Taylor invariant $t$ is now known for all knots in the table.
\item The smooth slice genus has been computed for all knots in the table:
Piccirillo~\cite{zbMATH07168645} has shown $\gsm(11n34) = 1$, and
for the 21 remaining knots in \cref{table:unknown}, $\gsm$ has been computed
in~\cite{zbMATH07445211,2112.14925}.
\item The topological slice genus $\gtop$ equals 2 for 12a1142 (see above),
and has been computed as 1 for 12a905 and 12n555~\cite{zbMATH07445211,2112.14925}.
So, $\gtop$ remains unknown just for the four knots
\[
12a244, 12a810, 12n549, 12n642.
\]
For each of those knots, $\gtop$ is either 1 or 2.
\item The algebraic genus $\galg$ of the remaining 52 knots has been computed in the latest version of~\cite{1905.08305},
so $\galg$ is now known for all knots in the table.
\end{itemize}
The values of $\gsm$ and $\gtop$ can be found on knotinfo~\cite{Knotinfo}, and the values of $t$ and $\galg$ in the online table~\cite{galg-table}.
\section{Taylor's lower bound}\label{sec:taylor}
The Seifert form $\theta$ of a knot $K$ yields several well-known lower bounds to the topological slice genus, namely the bounds coming from the Levine-Tristram signatures and the Fox-Milnor condition.
All of these lower bounds are subsumed by Taylor's lower bound $t(K)$ \cite{taylor}: let $a(\theta)$ be the maximal rank of an isotropic subgroup $U$ of $\mathbb{Z}^{\dim \theta}$,
i.e.~a subgroup on which $\theta|_{U\times U}$ is identically zero. Taylor's bound is then the following:
\[\gtop(K)\geq t(K) := \dim\theta/2 - a(\theta).\]
Since this bound has previously only been explicitly stated in the literature as a bound on the smooth slice genus, we briefly indicate why this is true.
The key ingredient is the existence of the higher-dimensional locally flat analogue of Seifert surfaces;
more precisely, given a locally flat surface $\Sigma\subset S^4$, the existence of a locally flat embedded compact oriented 3-manifold $X\subset S^4$ with boundary $\Sigma$.
Following  \cite[p. XXI]{ranicki}, such an $X$ may be constructed as follows:
let $\nu\Sigma\subset S^4$ denote an open tubular neighborhood of $\Sigma$. We want to
extend the projection $\partial\nu\Sigma \cong S^1\times \Sigma \to S^1$ onto the first factor to a function $\phi: S^4\setminus\nu\Sigma \to S^1$.
Then, because topological transversality holds (cf. \cite[Ch. 9]{quinn} and also \cite[Essay III, \S 1]{ks}),
there is a homotopy making $\phi$ topologically transverse to $1\in S^1$.
Thus $X=\phi^{-1}(1)$ is a locally flat 3-manifold with $X\cap\partial\nu\Sigma$ equal to a push-off of $\Sigma$. To obtain $\phi$, note that by Alexander duality and $S^1$ being a $K(\mathbb{Z}, 1$), we have
\[\mathbb{Z}\cong H_2(\nu \Sigma;\mathbb{Z}) \cong H^1(S^4\setminus \nu \Sigma;\mathbb{Z})\cong [S^4\setminus \nu \Sigma, S^1].\]
So one may take $\phi$ to be a generator for $[S^4\setminus \nu \Sigma, S^1]$.

Given this key ingredient, Taylor's bound can be obtained by a minor generalization to the proof that a topologically slice knot is algebraically slice (see e.g.~\cite[Ch. 8]{lickorish}).
Indeed, given a Seifert surface $F\subseteq S^3$ with Seifert form $\theta$ and a properly embedded locally flat surface $D\subseteq B^4$ cobounding $K$ we can form the closed embedded locally flat surface $\Sigma=F \cup D\subset S^4$.
Take $X$ to be an embedded 3-manifold with boundary $\Sigma$, and let $i:F\to X$ be the inclusion.
Standard homological and linear algebra arguments show that the kernel of $i_*: H_1(F) \to H_1(X)$ is an isotropic subgroup of $\theta$ of required rank.

Taylor's bound is much less well-known than the signatures and the Fox-Milnor condition, most likely because of two reasons: no algorithm to calculate $t(K)$ has been produced so far,
and it seems unlikely that $t(K)$ will be much stronger than all the Levine-Tristram signatures taken together.
Still, we exhibit the following computable lower bound to the slice genus coming from the Seifert form,
which has to the best knowledge of the authors not been stated in the literature before, although it was implicitly used by the first author in \cite{baader}.
This bound arises from Taylor's bound by a straight-forward application of the theory of quadratic forms, for which we refer
to any of the standard textbooks such as the one by Lam \cite{lam}.

\begin{thm}\label{thm:padic}
Let $K$ be a knot with a $2g$--dimensional Seifert form $\theta$. Denote by $\eta$ the integral quadratic form given by $\eta(v) = \theta(v,v)$.
Denote by $\disc \eta = (-1)^{g(2g-1)} \det(\eta)$ the discriminant of $\eta$ (note $|\disc\eta| = \det K$). If there is an odd prime $p$
such that the two following, equivalent conditions are satisfied, then $\gtop(K) \geq 2$:
\begin{enumerate}\renewcommand{\theenumi}{\roman{enumi}}
\item There is an integer $e\geq 0$, an integer $n$, and an integer $r$ not divisible by $p$,
such that $\disc \eta = p^{2e}(np + r^2)$, and the Hasse symbol of $\eta$ over $\mathbb{Q}_p$ is $-1$.
\item The form $\eta_{\mathbb{Q}_p}$ induced by $\eta$ over $\mathbb{Q}_p$ is Witt-equivalent to an anisotropic four-dimensional form.
\end{enumerate}
\end{thm}
\begin{proof}
An isotropic subgroup $U$ of rank $a(\theta)$ of $\theta$ gives rise to an isotropic subspace of dimension $a(\theta)$ of $\eta_{\mathbb{Q}_p}$.
If (ii) is satisfied, such a subspace has dimension at most $\dim\eta_{\mathbb{Q}_p} / 2 - 2 = g - 2$. Thus $a(\theta) \leq g - 2 \Rightarrow t(K) \geq 2 \Rightarrow \gtop(K) \geq 2$.
The equivalence of the conditions (i) and (ii) follows from the well-known fact that the Witt-class of a quadratic form over a local field
is determined by its discriminant and Hasse symbol.
\end{proof}
We used scripts \cite{przemyslaw} written for PARI/GP \cite{pari} to test for which knots this lower bound would be applicable, finding the six knots:
\\[-.5\baselineskip]

\begin{center}
\begin{tabular}{|l| p{10cm} |}
\hline
$\gtop(K)=2$ & $12a787$, $12n269$, $12n505$, $12n598$, $12n602$, $12n756$.\\
\hline
\end{tabular}
\end{center}
\vspace{.5\baselineskip}

\noindent For these knots, the topological slice genus is at least $2$, and, by the upper bounds already known, in fact equal to $2$.
For all of them, one takes the odd prime $p$ in \cref{thm:padic} to be $3$. The first two knots have discriminant $117 = 3^2\cdot(4\cdot 3 + 1)$, the last four discriminant $-99 = 3^2\cdot (-4\cdot 3 + 1)$.

Taylor's invariant $t(K)$ is not just a lower bound for the slice genus, but may be considered as knot invariant in its own right, with further noteworthy properties;
e.g., $t(K)$ is an algebraic concordance invariant and equals the minimal slice genus among knots with Seifert form $\theta$. As a side effect of our efforts to compute
slice genera, we have also computed $t(K)$ for all but five knots with up to 12 crossings.
Note that $t(K)$ is bounded below by the Levine-Tristram signatures, the Fox-Milnor condition and the bound from \cref{thm:padic}.
\begin{prop}
For all prime knots $K$ with up to 10 crossings, one has
\begin{equation}\label{eq:taylor=gtop}
t(K) = \gtop(K).
\end{equation}
\end{prop}
For 11 and 12 crossings, there are 16 exceptions, and 7 potential exceptions to \cref{eq:taylor=gtop}:
\begin{itemize}
\item The following 16 knots are algebraically slice, i.e.~$t(K) = 0$, but their topological sliceness has been obstructed by twisted Alexander polynomials \cite{hkl}
(in fact, they all have topological slice genus equal to $1$):
\begin{center}
$11n45$, $11n145$, $12a169$, $12a596$, $12n31$, $12n132$, $12n210$, $12n221$,\\
$12n224$, $12n264$, $12n536$, $12n681$, $12n731$, $12n812$, $12n813$, $12n841$.
\end{center}
\item We were unable to determine the Taylor invariant of the two knots
\[
12a1142, 12n549.
\]
For both of them, $t(K) \in \{1,2\}$.
For these knots, it is not known whether the topological slice genus is $1$ or $2$.
\item The Taylor invariant of the five knots
\begin{center}
$12a244$, $12a810$, $12a905$, $12n555$, $12n642$
\end{center}
can be shown to be $1$ by explicitly exhibiting an isotropic subgroup of the appropriate rank (cf. \nameref{app:taylor}).
For these knots, it is not known whether the topological slice genus is $1$ or $2$.
\end{itemize}
\section{An obstruction from Donaldson's theorem}\label{sec:Donaldson}
Now we state our obstruction from Donaldson's theorem. This obstruction has previously been applied to help find 2-bridge knots with differing smooth and topological slice genera \cite{fellermccoy}.
\begin{lem}\label{lem:donaldson}
Let $K$ be a knot with with a positive-definite $m \times m$ Goeritz matrix $G$. If $\sigma(K)\leq 0$ and $2\gsm(K)=-\sigma(K)$, then there is an $ (m-\sigma(K))\times m$ integer matrix $M$, such that $G=M^T M$.
\end{lem}
\begin{proof}
The double branched cover $\Sigma(K)$ bounds a 4-manifold $X$ with intersection form given by the matrix $G$ \cite{gordon1978signature}. It also bounds a smooth 4-manifold $Y$ with $b_2(Y)=2\gsm(K)$ and signature $\sigma(K)$ \cite{gordon1978signature,Kauffman76signature}. Suppose that $2\gsm(K)=-\sigma(K)$. This implies that the closed smooth 4-manifold $Z=X\cup (-Y)$ is positive definite. Therefore, $Z$ has intersection form isomorphic to the diagonal lattice $\mathbb{Z}^{m-\sigma(K)}$~\cite{donaldson87orientation}. The inclusion $X\hookrightarrow Z$ induces an injection $H_2(X) \hookrightarrow H_2(Z)$ and hence an embedding of intersection forms. This gives the desired matrix factorization.
\end{proof}
Using this lemma, we find that the following alternating knots cannot have smooth slice genus one, and thus all have smooth slice genus two. In each case, GAP's command \texttt{OrthogonalEmbeddings} was used to find minimal dimension matrix factorizations~\cite{GAP4}.
\\[-.5\baselineskip]

\begin{center}
\begin{tabular}{|l| p{10cm} |}
\hline
$\gsm(K)=2$ &
11a211, 12a244, 12a255, 12a414, 12a534, 12a542, 12a719, 12a810, 12a908, 12a1118, 12a1142, 12a1185\\
\hline
\end{tabular}
\end{center}
\section{The algebraic genus}\label{sec:galg}
The algebraic genus $\galg(K)$ is a knot invariant determined by the S-equivalence class of Seifert forms $\theta$ of a knot $K$, giving an upper bound for $\gtop(K)$.
It was recently introduced by Feller and the first author \cite{fellerlewark}; we refer to that paper for a detailed treatment,
and only briefly state a definition and some properties of $\galg$ here.
A subgroup $U \subset \mathbb{Z}^{\dim \theta}$ is called \emph{Alexander-trivial} if $\det(t\cdot \theta|_{U\times U} - \theta|_{U\times U}^{\top})$
is a unit in $\mathbb{Z}[t^{\pm 1}]$. Let $d$ be the maximal rank of an Alexander-trivial subgroup. Then the algebraic genus of $\theta$ is defined as
\[
\galg(\theta) = \frac{\dim \theta - d}{2},
\]
and the algebraic genus of $K$ is defined as the minimum algebraic genus of a Seifert form of $K$.
At the moment, no way is known to compute $\galg$ for a general knot. However, a randomized algorithm
as in \cite{BaaderFellerLewarkLiechti_15}, implemented in PARI/GP \cite{pari}, gives good upper
bounds for $g_{\text{alg}}$, and thus for $\gtop$. The upper bound for $\gtop$ given by $\galg$
subsumes the bound coming from the algebraic unknotting number $\ualg$ \cite{Murakami}:
\[
\ualg(K) \geq \galg(K) \geq \gtop(K).
\]

We found 19 knots for which the bound given by $\ualg$
is not strong enough, but the algebraic genus determines the previously unknown topological slice genus.
\nameref{app:galg} lists a Seifert matrix and a basis for an Alexander-trivial subgroup for each of those knots.
\\[-.5\baselineskip]

\begin{center}
\begin{tabular}{ | r | p{8.5cm} |}\hline
$\gtop(K) = 1$     & for ${11n80}$. \\\hline
$\gtop(12a_k) = 1$ & for $k \in$ \{{187}, {230}, {317}, {450}, {542}, {570}, {908}, {1118}, {1185}, {1189}, {1208}\}. \\\hline
$\gtop(12n_k) = 1$ & for $k \in$ \{{52}, {63}, {225}, {558}, {665}, {886}\}. \\\hline
$\gtop(K) = 2$     & for $12n276$. \\\hline
\end{tabular}
\end{center}
\vspace{.5\baselineskip}
\begin{figure}[t]
\includegraphics[width=.4\textwidth]{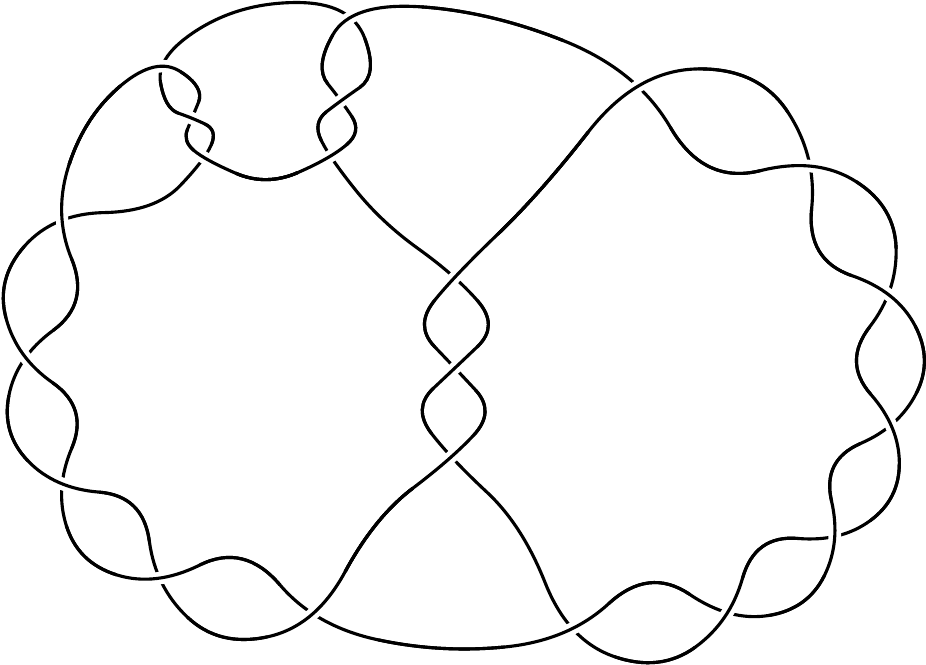}
\caption{A prime knot $K$ with $\gtop(K) = 0$, $\galg(K) > 0$, $\gsm(K) > 0$ (drawn with knotscape \cite{knotscape}).}\label{fig:1}
\end{figure}

Note that the minimum of the algebraic genus and the smooth slice genus is a very efficient upper bound for the topological slice genus of small knots:
\begin{prop}
For all prime knots with up to 11 crossings, one has
\begin{equation}\label{eq:galg}
\gtop(K) = \min \{ \galg(K), \gsm(K) \}
\end{equation}
\end{prop}
Indeed, \cref{eq:galg} even holds for all prime knots with up to 12 crossings with the potential exceptions of the 7 knots for which $\gtop$ is unknown (see \cref{table:unknown}).
Note that \cref{eq:galg} need not hold for higher crossing numbers. For example, the knot in \cref{fig:1} is topologically but not smoothly slice (it is concordant to the $(-5,3,-7)$--pretzel knot) and has non-trivial Alexander polynomial (and thus non-zero algebraic genus).

Just as the Taylor invariant, the algebraic genus may be considered as a knot invariant of independent interest. We were able to determine the algebraic
genus for all knots with up to 12 crossings, except for the following 54 knots, all of which
have Alexander polynomial not equal to $1$, but algebraic unknotting number equal to $2$, which implies $\galg \in \{1,2\}$:
\\[-.5\baselineskip]

\begin{center}
\begin{tabular}{|  r  @{\ }p{9.1cm} |} \hline
        & $9_{37}$, $9_{48}$, $10_{74}$, $11a135$, $11a155$, $11a173$, $11a352$, $11n71$, $11n75$, $11n167$ \\\hline
$12a_k$ for $k \in$ & \{164, 166, 177, 244, 265, 298, 396, 413, 493, 503, 735, 769, 810, 873, 895, 905, 1013, 1047, 1142, 1168, 1203, 1211, 1221, 1222, 1225, 1226, 1229, 1230, 1248, 1260, 1283, 1288\} \\\hline
$12n_k$ for $k \in$ & \{334, 379, 388, 460, 480, 495, 549, 583, 737, 813, 846, 869\} \\\hline
\end{tabular}
\end{center}
\vspace{.5\baselineskip}

For all other 2923 prime knots with up 12 crossings, the algebraic genus equals the maximum of two of its lower bounds:
the Taylor invariant and $\lceil \ualg/2\rceil$. A sharp upper bound for the algebraic genus is given in 2324 cases by the algebraic unknotting number, i.e.~$\galg = \ualg$;
in the other 599 cases, we explicitly found an Alexander-trivial subgroup of sufficient rank (for those knots, $\galg = \ualg - 1$).
To avoid overly bloating the appendix, bases for those subgroups are included in a separate text file with the arXiv-version of this article.

\section{Genus one concordances}\label{sec:genusone}
\begin{lem}\label{lem:diagoperations}
Let $K$ and $K'$ be knots in $S^3$. If $K'$ can be obtained from $K$ by one of the following operations:
\begin{enumerate}[label=(\roman*)]
\item changing a single crossing;
\item changing a positive and a negative crossing; or
\item taking oriented resolutions of two crossings,
\end{enumerate}
then $|\gsm(K)-\gsm(K')|\leq 1$ and $|\gtop(K)-\gtop(K')|\leq 1$.
\end{lem}
\begin{proof}[Proof (sketch).]
In all three cases $K'$ can be obtained from $K$ by adjoining two oriented bands. This shows that in each case, there is a smoothly embedded twice-punctured torus $T\subset S^3 \times [0,1]$ with $\partial T= K\times \{0\} \cup K' \times \{1\}$. Since we can glue $T$ to any properly embedded (smooth or locally flat) surface $F\subset B^4$ with boundary $K$ or $K'$, we see that the desired inequalities hold.
\end{proof}
\afterpage{%
\renewcommand{\arraystretch}{1}
\begin{table}
\begin{center}
\begin{small}
\begin{tabular}{ | l | p{10cm} |}
\hline
$\gsm(11a_k)=1$ & for $k\in \{$1, 102, 107, 108, 109, 110, 111, 118, 119, 125, 126, 128, 13, 130, 131, 132, 133, 134, 135, 137, 14, 141, 145, 147, 148, 15, 151, 152, 153, 154, 155, 156, 157, 158, 159, 16, 162, 163, 166, 17, 170, 171, 172, 173, 174, 175, 176, 178, 18, 181, 183, 185, 188, 19, 193, 197, 199, 202, 205, 21, 214, 217, 218, 219, 221, 228, 229, 23, 231, 232, 239, 24, 248, 249, 25, 251, 252, 253, 254, 258, 26, 262, 265, 268, 269, 27, 270, 271, 273, 274, 277, 278, 279, 281, 284, 285, 288, 29, 294, 296, 297, 3, 30, 301, 303, 305, 312, 313, 314, 315, 317, 32, 322, 323, 324, 325, 327, 33, 331, 332, 333, 347, 349, 350, 352, 37, 38, 39, 4, 42, 44, 45, 46, 47, 50, 51, 52, 54, 55, 57, 59, 6, 61, 65, 66, 67, 68, 7, 72, 75, 76, 79, 81, 84, 85, 89, 90, 92, 93, 97, 99$\}$. \\
      \hline

$\gsm(11n_k)=1$ & for $k\in \{$102, 11, 112, 113, 115, 117, 119, 120, 127, 128, 129, 138, 140, 142, 146, 148, 15, 150, 155, 157, 160, 161, 162, 163, 165, 166, 167, 168, 17, 170, 177, 178, 179, 182, 24, 29, 3, 32, 33, 36, 40, 44, 46, 5, 51, 54, 58, 6, 60, 65, 66, 7, 79, 91, 92, 94, 98, 99$\}$.
      \\ \hline

$\gsm(11a_k)=2$ & for $k\in \{$105, 144, 161, 20, 293, 304, 346, 49, 53, 60, 63, 64, 83$\}$. \\
\hline
$\gsm(11n_k)=2$ & for $k\in \{$133, 137, 173, 23, 30$\}$.
      \\ \hline

$\gsm(12a_k)=1$ & for $k\in \{$4, 10, 39, 45, 49, 50, 65, 66, 76, 86, 89, 103, 104, 108, 120, 125, 127, 128, 129, 135, 150, 161, 163, 164, 166, 168, 175, 177, 178, 181, 194, 196, 200, 204, 212, 247, 259, 260, 265, 291, 292, 296, 298, 302, 312, 327, 338, 339, 342, 353, 354, 357, 364, 372, 376, 379, 380, 381, 395, 396, 399, 400, 412\footnotemark,413, 423, 424, 434, 436, 438, 448, 449, 454, 459, 462, 463, 465, 468, 481, 482, 489, 493, 494, 496, 503, 505, 544, 545, 549, 554, 564, 581, 582, 597, 598, 601, 609, 621, 634, 639, 642, 643, 644, 649, 665, 668, 669, 677, 680, 684, 687, 689, 690, 691, 704, 706, 735, 749, 750, 752, 757, 767, 769, 771, 783, 784, 789, 791, 815, 816, 818, 824, 825, 826, 827, 833, 835, 842, 845, 852, 853, 862, 870, 871, 873, 878, 886, 895, 896, 898, 899, 901, 911, 912, 914, 916, 921, 939, 940, 941, 942, 957, 971, 981, 989, 999, 1000, 1012, 1014, 1016, 1025, 1028, 1039, 1040, 1050, 1061, 1066, 1085, 1095, 1103, 1109, 1110, 1124, 1127, 1138, 1145, 1147, 1148, 1149, 1150, 1151, 1160, 1163, 1165, 1171, 1174, 1175, 1179, 1194, 1200, 1201, 1205, 1226, 1254, 1256, 1259, 1275, 1279, 1281, 1282, 1284, 1285, 1288$\}$. \\
      \hline
$\gsm(12n_k)=1$ &  for $k\in \{$47, 60, 61, 75, 80, 84, 92, 101, 109, 115, 116, 118, 137, 140, 147, 157, 159, 167, 171, 176, 192, 193, 197, 200, 202, 206, 208, 211, 212, 216, 219, 227, 236, 247, 248, 253, 258, 260, 267, 270, 291, 304, 307, 324, 334, 345, 351, 359, 376, 379, 383, 388, 391, 396, 409, 410, 411, 439, 442, 443, 451, 454, 456, 460, 469, 475, 480, 489, 495, 500, 514, 519, 520, 522, 524, 525, 531, 532, 537, 543, 554, 564, 569, 577, 583, 595\footnotemark[\value{footnote}], 596, 601, 606, 608, 621, 630, 631, 672, 673, 675, 678, 681, 685, 699, 701, 717, 726, 730, 735, 737, 742, 759, 769, 777, 783, 794, 797, 804, 805, 808, 809, 811, 813, 814, 815, 818, 822, 824, 829, 833, 844, 846, 854, 855, 856, 859, 861, 862, 869, 873, 875$\}$.
\\ \hline
$\gsm(12a_k)=2$ & for $k\in \{$75, 147, 148, 160, 167, 193, 195, 231, 289, 311, 370, 375, 580, 692, 693, 725, 730, 741, 812, 841, 967, 983, 988, 1115, 1116, 1278, 1286$\}$.
\\ \hline
$\gsm(12n_k)=2$ & for $k\in \{$113, 190, 204, 233, 441, 496, 626, 698, 700, 707, 734, 796, 863, 867$\}$.
      \\ \hline
\end{tabular}
\end{small}
\end{center}
\vspace{.5ex}
\caption{Knots whose smooth slice genus could be determined by genus one cobordisms.}
\label{table:g1}
\end{table}
\footnotetext{\label{note1}$12a412$ and $12n595$ are both Gordian distance one from a knot with the same Jones polynomial as $11n50$, $11n132$ and $12n414$ which are all slice, and $11n133$, which is not. It can be verified that the knot obtained in both cases is, in fact $11n50$, confirming that the smooth slice genus is one.}
\clearpage}
The above lemma was used to generate upper bounds for the slice genus as follows. For a given knot $K$, we take a diagram $D$ and obtain a new diagram $D'$ by performing one of the three given operations. By considering the Jones polynomial and the crossing number of $D'$, we obtain a small set of possibilities, $S$, for the knot represented by $D'$. This gives the following upper bound on the smooth slice genus:
\begin{equation}\label{eq:bound}
\gsm(K)\leq 1 + \max_{K' \in S} \gsm(K').
\end{equation}

\begin{rem}
In practice, we found that the combination of the Jones polynomials and the bound on crossing number was often sufficient to identify $D'$ exactly.
However, even when the set $S$ contains more than one knot, we often found that all the knots in $S$ had the same smooth slice genus.
\end{rem}

Using a computer, we applied the above method to each knot where the smooth slice genus was not listed on KnotInfo \cite{Knotinfo}. Comparing the resulting upper bounds with previously known lower bounds allowed to deduce the exact value for the knots listed below. In \nameref{app:moves}, we state the operation performed and the resulting set $S$ which gave the desired upper bound for each knot.

\subsection{The smooth slice genus of 11- and 12-crossing knots}\label{sec:smoothg11}
Using the methods outlined above, the value of the smooth genus can be determined for the knots listed in \cref{table:g1}. In all cases, this means that the smooth genus is equal to the previously known lower bound.
Although it does not determine the value completely, we also obtain a new upper bound for $\gsm(12n555)\leq 2$. This is done by observing $12n555$ can be transformed into $9_{48}$, which has $\gsm(9_{48})=1$, by a crossing change.

\subsection{The topological slice genus of 10-crossing knots}\label{sec:10top}
The only 10-crossing knots for which the topological genera are not known are $10_{152}$ and $10_{154}$. We can determine for these by obtaining them by a crossing change from knots for which the topological genus is already known.\\[-.5\baselineskip]

\noindent\hfill\begin{tabular}{ | l | p{9.5cm} |}
    \hline
    $\gtop(10_{152})= 3$ &
     We can obtain $10_{152}$ by performing a single crossing change in a 12-crossing diagram for $12n750$. Since $\gtop(12n750)=2$ and $\gtop(10_{152})\geq 3$, this implies that $\gtop(10_{152})= 3$.
    \\ \hline
    $\gtop(10_{154})= 2$ &
    We can obtain $10_{154}$ by performing a single crossing change in a 12-crossing diagram for $12n321$. Since $\gtop(12n321)=1$ and $\gtop(10_{154})\geq 2$, this implies that $\gtop(10_{154})= 2$.
    \\ \hline
\end{tabular}\hfill\mbox{}
\vspace{.5\baselineskip}

Note that this uses the topological genera of $12n321$ and $12n750$ which were recently determined by Feller \cite{feller15degree}.

\section*{Acknowledgments}
The second author would like to thank his supervisor, Brendan Owens, for his continued guidance. Both authors would like to thank Peter Feller for the conversations that led to the undertaking of this project.

\bibliographystyle{plain}
\bibliography{gencomp}

\begin{thebibliography}{10}

\bibitem{BaaderFellerLewarkLiechti_15}
S.~Baader, P.~Feller, L.~Lewark, and L.~Liechti.
\newblock On the topological 4-genus of torus knots.
\newblock {\em Trans. Am. Math. Soc.}, 370(4):2639--2656, 2018.

\bibitem{baader}
S.~Baader and L.~Lewark.
\newblock The stable {{\(4\)}}-genus of alternating knots.
\newblock {\em Asian J. Math.}, 21(6):1183--1190, 2017.

\bibitem{2112.14925}
M.~Brittenham and S.~Hermiller.
\newblock The smooth 4-genus of (the rest of) the prime knots through 12
  crossings.
\newblock {\href{http://arxiv.org/abs/2112.14925}{arXiv:2112.14925}}, 2021.

\bibitem{CassonGordon_86}
A.~J. Casson and C.~McA. Gordon.
\newblock Cobordism of classical knots.
\newblock In {\em \`{A} la recherche de la topologie perdue}, volume~62 of {\em
  Progr. Math.}, pages 181--199. Birkh\"auser Boston, Boston, MA, 1986.
\newblock With an appendix by P. M. Gilmer.

\bibitem{Cha_08_TopMinGenusandL2}
J.~C. Cha.
\newblock Topological minimal genus and {$L^2$}-signatures.
\newblock {\em Algebr. Geom. Topol.}, 8(2):885--909, 2008.

\bibitem{Knotinfo}
J.~C. Cha and C.~Livingston.
\newblock Knotinfo: {T}able of {K}not {I}nvariants.
\newblock \url{http://www.indiana.edu/~knotinfo}, 2016.

\bibitem{donaldson87orientation}
S.~K. Donaldson.
\newblock The orientation of {Y}ang-{M}ills moduli spaces and {$4$}-manifold
  topology.
\newblock {\em J. Differential Geom.}, 26(3):397--428, 1987.

\bibitem{feller15degree}
P.~Feller.
\newblock The degree of the {A}lexander polynomial is an upper bound for the
  topological slice genus.
\newblock {\em Geom. Topol.}, 20(3):1763--1771, 2016.

\bibitem{fellerlewark}
P.~Feller and L.~Lewark.
\newblock On classical upper bounds for slice genera.
\newblock {\em Sel. Math., New Ser.}, 24(5):4885--4916, 2018.

\bibitem{1905.08305}
P.~Feller and L.~Lewark.
\newblock Balanced algebraic unknotting, linking forms, and surfaces in three-
  and four-space.
\newblock {\href{http://arxiv.org/abs/1905.08305}{arXiv:1905.08305}}, accepted
  for publication by J. Differential Geom., 2019.

\bibitem{fellermccoy}
P.~Feller and D.~McCoy.
\newblock On 2-bridge knots with differing smooth and topological slice genera.
\newblock {\em Proc. Amer. Math. Soc.}, 144(12):5435--5442, 2016.

\bibitem{Freedman_82_TheTopOfFour-dimensionalManifolds}
M.~H. Freedman.
\newblock The topology of four-dimensional manifolds.
\newblock {\em J. Differential Geom.}, 17(3):357--453, 1982.

\bibitem{quinn}
M.~H. Freedman and F.~Quinn.
\newblock {\em Topology of 4-manifolds}, volume~39 of {\em Princeton
  Mathematical Series}.
\newblock Princeton University Press, Princeton, NJ, 1990.

\bibitem{GAP4}
The GAP~Group.
\newblock {\em {GAP -- Groups, Algorithms, and Programming, Version 4.7.8}},
  2015.

\bibitem{GompfStip}
R.~Gompf and A.~Stipsicz.
\newblock {\em {$4$}-manifolds and {K}irby calculus}, volume~20 of {\em
  Graduate Studies in Mathematics}.
\newblock American Mathematical Society, Providence, RI, 1999.

\bibitem{gordon1978signature}
C.~McA. Gordon and R.~A. Litherland.
\newblock On the signature of a link.
\newblock {\em Invent. Math.}, 47(1):53--69, 1978.

\bibitem{hkl}
C.~Herald, P.~Kirk, and C.~Livingston.
\newblock Metabelian representations, twisted {A}lexander polynomials, knot
  slicing, and mutation.
\newblock {\em Math. Z.}, 265(4):925--949, 2010.

\bibitem{knotscape}
J.~Hoste and M.~Thistlethwaite.
\newblock Knotscape (version 1.01), 1999.
\newblock Computer program.

\bibitem{zbMATH07445211}
L.~P. Karageorghis and F.~Swenton.
\newblock Determining the doubly slice genera of prime knots with up to
  {{\(12\)}} crossings.
\newblock {\em J. Knot Theory Ramifications}, 30(8):17, 2021.
\newblock Id/No 2150057.

\bibitem{Kauffman76signature}
L.~H. Kauffman and L.~R. Taylor.
\newblock Signature of links.
\newblock {\em Trans. Amer. Math. Soc.}, 216:351--365, 1976.

\bibitem{ks}
R.~C. Kirby and L.~C. Siebenmann.
\newblock {\em Foundational essays on topological manifolds, smoothings, and
  triangulations}.
\newblock Princeton University Press, Princeton, N.J.; University of Tokyo
  Press, Tokyo, 1977.
\newblock With notes by John Milnor and Michael Atiyah, Annals of Mathematics
  Studies, No. 88.

\bibitem{przemyslaw}
P.~Koprowski and A.~Czoga{\l{}}a.
\newblock Computing with quadratic forms over number fields.
\newblock {\href{http://arxiv.org/abs/1304.0708}{arXiv:1304.0708}}, 2013.

\bibitem{KronheimerMrowka_Gaugetheoryforemb}
P.~B. Kronheimer and T.~S. Mrowka.
\newblock Gauge theory for embedded surfaces. {I}.
\newblock {\em Topology}, 32(4):773--826, 1993.

\bibitem{lam}
T.~Y. Lam.
\newblock {\em Introduction to quadratic forms over fields}, volume~67 of {\em
  Graduate Studies in Mathematics}.
\newblock American Mathematical Society, Providence, RI, 2005.

\bibitem{Levine1969}
J.~Levine.
\newblock Invariants of knot cobordism.
\newblock {\em Invent. Math.}, 8(2):98--110, 1969.

\bibitem{galg-table}
L.~Lewark.
\newblock Table of values of $g_{\text{alg}}$ and the {T}aylor invariant, 2023.
\newblock \url{https://github.com/LLewark/galg-taylor-table}.

\bibitem{lickorish}
W.~B.~Raymond Lickorish.
\newblock {\em An introduction to knot theory}, volume 175 of {\em Graduate
  Texts in Mathematics}.
\newblock Springer-Verlag, New York, 1997.

\bibitem{Murakami}
H.~Murakami.
\newblock Algebraic unknotting operation.
\newblock In {\em Proceedings of the {S}econd {S}oviet-{J}apan {J}oint
  {S}ymposium of {T}opology ({K}habarovsk, 1989)}, volume~8 of {\em Questions
  Answers Gen. Topology}, pages 283--292, 1990.

\bibitem{orevkov}
Stepan Orevkov.
\newblock Private communication.
\newblock February 2023.

\bibitem{MR2026543}
P.~Ozsv\'ath and Z.~Szab\'o.
\newblock Knot {F}loer homology and the four-ball genus.
\newblock {\em Geom. Topol.}, 7:615--639, 2003.

\bibitem{pari}
The PARI~Group, Bordeaux.
\newblock {\em PARI/GP, version 2.7.4}, 2015.
\newblock Programming library available from
  \url{http://pari.math.u-bordeaux.fr/}.

\bibitem{zbMATH07168645}
Lisa Piccirillo.
\newblock The {Conway} knot is not slice.
\newblock {\em Ann. Math. (2)}, 191(2):581--591, 2020.

\bibitem{ranicki}
A.~Ranicki.
\newblock {\em High-dimensional knot theory}.
\newblock Springer Monographs in Mathematics. Springer-Verlag, New York, 1998.
\newblock Algebraic surgery in codimension 2, With an appendix by Elmar
  Winkelnkemper.

\bibitem{rasmussen10khovanov}
J.~Rasmussen.
\newblock Khovanov homology and the slice genus.
\newblock {\em Invent. Math.}, 182(2):419--447, 2010.

\bibitem{taylor}
L.~R. Taylor.
\newblock On the genera of knots.
\newblock In {\em Topology of low-dimensional manifolds ({P}roc. {S}econd
  {S}ussex {C}onf., {C}helwood {G}ate, 1977)}, volume 722 of {\em Lecture Notes
  in Math.}, pages 144--154. Springer, Berlin, 1979.

\bibitem{tristram_1969}
A.~G. Tristram.
\newblock Some cobordism invariants for links.
\newblock {\em Proc. Cambridge Philos. Soc.}, 66(2):251–264, 1969.

\end{thebibliography}
\addresseshere
\twocolumn
\begin{appendix}
\renewcommand{\arraystretch}{1}
\section*{Appendix A}\label{app:moves}
\begin{small}
For each knot $K$ in \cref{table:g1}, we state an operation from \cref{lem:diagoperations} along with the resulting family $S$, which can be used to give the required bound on $\gsm(K)$. In all cases we used only the diagram corresponding to the PD notation listed on KnotInfo \cite{Knotinfo}.
In this table,
\begin{itemize}
\item`cc'=`crossing change',
\item`+-cc'=`a positive and a negative crossing change'
\item and `res'=`oriented resolutions of\linebreak two crossings'.\\[\baselineskip]
\end{itemize}

\tablehead{\hline Knot & \tiny Operation & $S$\\ \hline}
\tabletail{\hline}

\newcolumntype{P}[1]{>{\centering\arraybackslash}p{#1}}
\begin{supertabular}{|p{1.2cm} | P{1.2cm} | p{2.6cm} |}
11a1 & +-cc & [U]\\
11a3 & +-cc & [U]\\
11a4 & +-cc & [U]\\
11a6 & +-cc & [U]\\
11a7 & +-cc & [U]\\
11a13 & +-cc & [6\_1]\\
11a14 & +-cc & [3\_1\#-(3\_1)]\\
11a15 & +-cc & [3\_1\#-(3\_1)]\\
11a16 & +-cc & [6\_1]\\
11a17 & +-cc & [6\_1]\\
11a18 & cc & [8\_8]\\
11a19 & +-cc & [U]\\
11a20 & +-cc & [3\_1]\\
11a21 & +-cc & [U]\\
11a23 & +-cc & [U]\\
11a24 & +-cc & [6\_1]\\
11a25 & +-cc & [U]\\
11a26 & +-cc & [6\_1]\\
11a27 & +-cc & [U]\\
11a29 & +-cc & [6\_1]\\
11a30 & +-cc & [U]\\
11a32 & +-cc & [U]\\
11a33 & +-cc & [6\_1]\\
11a37 & +-cc & [6\_1]\\
11a38 & +-cc & [6\_1]\\
11a39 & +-cc & [6\_1]\\
11a42 & +-cc & [U]\\
11a44 & +-cc & [U]\\
11a45 & +-cc & [6\_1]\\
11a46 & +-cc & [U]\\
11a47 & +-cc & [U]\\
11a49 & +-cc & [5\_2]\\
11a50 & +-cc & [U]\\
11a51 & cc & [8\_9, 4\_1\#4\_1]\\
11a52 & +-cc & [U]\\
11a53 & +-cc & [3\_1]\\
11a54 & +-cc & [U]\\
11a55 & +-cc & [U]\\
11a57 & +-cc & [U]\\
11a59 & +-cc & [6\_1]\\
11a60 & +-cc & [5\_2]\\
11a61 & res & [6\_1]\\
11a63 & +-cc & [5\_2]\\
11a64 & +-cc & [3\_1]\\
11a65 & +-cc & [U]\\
11a66 & +-cc & [U]\\
11a67 & +-cc & [U]\\
11a68 & +-cc & [U]\\
11a72 & +-cc & [3\_1\#-(3\_1)]\\
11a75 & +-cc & [6\_1]\\
11a76 & +-cc & [U]\\
11a79 & +-cc & [U]\\
11a81 & +-cc & [U]\\
11a83 & +-cc & [3\_1]\\
11a84 & +-cc & [6\_1]\\
11a85 & +-cc & [U]\\
11a89 & +-cc & [6\_1]\\
11a90 & +-cc & [U]\\
11a92 & +-cc & [U]\\
11a93 & +-cc & [U]\\
11a97 & +-cc & [6\_1]\\
11a99 & +-cc & [6\_1]\\
11a102 & +-cc & [6\_1]\\
11a105 & +-cc & [3\_1]\\
11a107 & +-cc & [U]\\
11a108 & +-cc & [U]\\
11a109 & +-cc & [3\_1\#-(3\_1)]\\
11a110 & +-cc & [6\_1]\\
11a111 & +-cc & [U]\\
11a118 & +-cc & [U]\\
11a119 & +-cc & [6\_1]\\
11a125 & +-cc & [U]\\
11a126 & +-cc & [3\_1\#-(3\_1)]\\
11a128 & +-cc & [6\_1]\\
11a130 & +-cc & [U]\\
11a131 & +-cc & [U]\\
11a132 & +-cc & [U]\\
11a133 & +-cc & [U]\\
11a134 & cc & [8\_8]\\
11a135 & +-cc & [6\_1]\\
11a137 & +-cc & [6\_1]\\
11a141 & +-cc & [6\_1]\\
11a144 & +-cc & [7\_2]\\
11a145 & res & [6\_1]\\
11a147 & +-cc & [U]\\
11a148 & res & [6\_1]\\
11a151 & +-cc & [8\_20]\\
11a152 & +-cc & [U]\\
11a153 & +-cc & [U]\\
11a154 & cc & [6\_1]\\
11a155 & cc & [8\_20]\\
11a156 & +-cc & [8\_20]\\
11a157 & +-cc & [U]\\
11a158 & +-cc & [U]\\
11a159 & +-cc & [U]\\
11a161 & +-cc & [7\_6]\\
11a162 & +-cc & [U]\\
11a163 & +-cc & [3\_1\#-(3\_1)]\\
11a166 & +-cc & [U]\\
11a170 & +-cc & [U]\\
11a171 & +-cc & [U]\\
11a172 & +-cc & [U]\\
11a173 & cc & [8\_20]\\
11a174 & +-cc & [U]\\
11a175 & +-cc & [U]\\
11a176 & +-cc & [U]\\
11a178 & +-cc & [U]\\
11a181 & +-cc & [6\_1]\\
11a183 & +-cc & [U]\\
11a185 & +-cc & [U]\\
11a188 & +-cc & [6\_1]\\
11a193 & +-cc & [U]\\
11a197 & cc & [8\_8]\\
11a199 & +-cc & [6\_1]\\
11a202 & +-cc & [6\_1]\\
11a205 & +-cc & [U]\\
11a214 & +-cc & [6\_1]\\
11a217 & +-cc & [U]\\
11a218 & +-cc & [U]\\
11a219 & res & [6\_1]\\
11a221 & +-cc & [8\_20]\\
11a228 & +-cc & [U]\\
11a229 & +-cc & [U]\\
11a231 & +-cc & [U]\\
11a232 & +-cc & [U]\\
11a239 & +-cc & [U]\\
11a248 & +-cc & [U]\\
11a249 & +-cc & [U]\\
11a251 & +-cc & [3\_1\#-(3\_1)]\\
11a252 & +-cc & [3\_1\#-(3\_1)]\\
11a253 & +-cc & [6\_1]\\
11a254 & +-cc & [3\_1\#-(3\_1)]\\
11a258 & +-cc & [6\_1]\\
11a262 & +-cc & [U]\\
11a265 & +-cc & [U]\\
11a268 & +-cc & [U]\\
11a269 & +-cc & [U]\\
11a270 & +-cc & [U]\\
11a271 & +-cc & [U]\\
11a273 & +-cc & [U]\\
11a274 & +-cc & [U]\\
11a277 & +-cc & [U]\\
11a278 & +-cc & [6\_1]\\
11a279 & +-cc & [U]\\
11a281 & +-cc & [6\_1]\\
11a284 & +-cc & [U]\\
11a285 & +-cc & [U]\\
11a288 & +-cc & [U]\\
11a293 & +-cc & [5\_2]\\
11a294 & +-cc & [U]\\
11a296 & res & [6\_1]\\
11a297 & +-cc & [U]\\
11a301 & +-cc & [U]\\
11a303 & +-cc & [U]\\
11a304 & +-cc & [5\_2]\\
11a305 & +-cc & [U]\\
11a312 & res & [3\_1\#-(3\_1)]\\
11a313 & +-cc & [U]\\
11a314 & +-cc & [U]\\
11a315 & +-cc & [U]\\
11a317 & +-cc & [U]\\
11a322 & +-cc & [U]\\
11a323 & +-cc & [6\_1]\\
11a324 & res & [6\_1]\\
11a325 & +-cc & [U]\\
11a327 & cc & [8\_20]\\
11a331 & +-cc & [U]\\
11a332 & +-cc & [U]\\
11a333 & +-cc & [U]\\
11a346 & +-cc & [3\_1]\\
11a347 & +-cc & [U]\\
11a349 & +-cc & [U]\\
11a350 & +-cc & [6\_1]\\
11a352 & cc & [6\_1]\\
11n3 & +-cc & [U]\\
11n5 & +-cc & [U]\\
11n6 & +-cc & [U]\\
11n7 & +-cc & [U]\\
11n11 & +-cc & [U]\\
11n15 & +-cc & [U]\\
11n17 & res & [U]\\
11n23 & +-cc & [5\_2]\\
11n24 & +-cc & [U]\\
11n29 & +-cc & [U]\\
11n30 & +-cc & [5\_2]\\
11n32 & +-cc & [6\_1]\\
11n33 & +-cc & [6\_1]\\
11n36 & +-cc & [U]\\
11n40 & res & [U]\\
11n44 & +-cc & [U]\\
11n46 & res & [U]\\
11n51 & +-cc & [U]\\
11n54 & res & [U]\\
11n58 & +-cc & [3\_1\#-(3\_1)]\\
11n60 & +-cc & [U]\\
11n65 & cc & [3\_1\#-(3\_1)]\\
11n66 & +-cc & [U]\\
11n79 & +-cc & [6\_1]\\
11n91 & res & [U]\\
11n92 & +-cc & [U]\\
11n94 & +-cc & [U]\\
11n98 & cc & [3\_1\#-(3\_1)]\\
11n99 & res & [U]\\
11n102 & +-cc & [U]\\
11n112 & +-cc & [U]\\
11n113 & res & [U]\\
11n115 & +-cc & [U]\\
11n117 & +-cc & [6\_1]\\
11n119 & +-cc & [6\_1]\\
11n120 & +-cc & [U]\\
11n127 & res & [U]\\
11n128 & +-cc & [6\_1]\\
11n129 & +-cc & [U]\\
11n133 & +-cc & [3\_1]\\
11n137 & res & [6\_2]\\
11n138 & +-cc & [6\_1]\\
11n140 & res & [6\_1]\\
11n142 & +-cc & [6\_1]\\
11n146 & res & [U]\\
11n148 & res & [8\_9, 4\_1\#4\_1]\\
11n150 & +-cc & [U]\\
11n155 & +-cc & [6\_1]\\
11n157 & +-cc & [U]\\
11n160 & +-cc & [U]\\
11n161 & +-cc & [U]\\
11n162 & res & [U]\\
11n163 & +-cc & [U]\\
11n165 & +-cc & [U]\\
11n166 & +-cc & [U]\\
11n167 & res & [6\_1]\\
11n168 & +-cc & [U]\\
11n170 & res & [6\_1]\\
11n173 & +-cc & [5\_2]\\
11n177 & +-cc & [U]\\
11n178 & res & [U]\\
11n179 & +-cc & [U]\\
11n182 & +-cc & [U]\\
12a4 & +-cc & [6\_1]\\
12a10 & +-cc & [6\_1]\\
12a39 & +-cc & [8\_20]\\
12a45 & +-cc & [3\_1\#-(3\_1)]\\
12a49 & res & [8\_8]\\
12a50 & +-cc & [8\_20]\\
12a65 & +-cc & [3\_1\#-(3\_1)]\\
12a66 & +-cc & [3\_1\#-(3\_1)]\\
12a75 & +-cc & [7\_4]\\
12a76 & res & [6\_1]\\
12a86 & res & [8\_8]\\
12a89 & cc & [8\_8]\\
12a103 & +-cc & [6\_1]\\
12a104 & cc & [8\_8]\\
12a108 & +-cc & [8\_20]\\
12a120 & +-cc & [8\_20]\\
12a125 & +-cc & [6\_1]\\
12a127 & +-cc & [6\_1]\\
12a128 & +-cc & [6\_1]\\
12a129 & +-cc & [3\_1\#-(3\_1)]\\
12a135 & +-cc & [6\_1]\\
12a147 & +-cc & [7\_4]\\
12a148 & +-cc & [7\_4]\\
12a150 & res & [8\_8]\\
12a160 & cc & [10\_77]\\
12a161 & res & [8\_8]\\
12a163 & res & [8\_8]\\
12a164 & +-cc & [8\_20]\\
12a166 & +-cc & [8\_20]\\
12a167 & cc & [3\_1\#-(3\_1)\#3\_1]\\
12a168 & cc & [3\_1\#-(3\_1)]\\
12a175 & +-cc & [6\_1]\\
12a177 & +-cc & [6\_1]\\
12a178 & +-cc & [6\_1]\\
12a181 & +-cc & [6\_1]\\
12a193 & +-cc & [8\_6]\\
12a194 & res & [10\_22, 10\_35]\\
12a195 & +-cc & [8\_6]\\
12a196 & cc & [10\_22, 10\_35]\\
12a200 & +-cc & [9\_46]\\
12a204 & +-cc & [6\_1]\\
12a212 & +-cc & [8\_20]\\
12a231 & +-cc & [7\_4]\\
12a247 & cc & [10\_129, 8\_8]\\
12a259 & +-cc & [6\_1]\\
12a260 & +-cc & [3\_1\#-(3\_1)]\\
12a265 & +-cc & [6\_1]\\
12a289 & +-cc & [7\_4]\\
12a291 & +-cc & [6\_1]\\
12a292 & +-cc & [8\_20]\\
12a296 & +-cc & [6\_1]\\
12a298 & +-cc & [8\_20]\\
12a302 & cc & [8\_8]\\
12a311 & +-cc & [7\_4]\\
12a312 & res & [6\_1]\\
12a327 & cc & [8\_8]\\
12a338 & +-cc & [6\_1]\\
12a339 & res & [6\_1]\\
12a342 & +-cc & [6\_1]\\
12a353 & cc & [8\_8]\\
12a354 & res & [8\_8]\\
12a357 & +-cc & [6\_1]\\
12a364 & +-cc & [8\_20]\\
12a370 & +-cc & [7\_4]\\
12a372 & res & [8\_8]\\
12a375 & res & [3\_1\#-(3\_1)\#3\_1]\\
12a376 & res & [3\_1\#-(3\_1)]\\
12a379 & res & [3\_1\#-(3\_1)]\\
12a380 & cc & [10\_129, 8\_8]\\
12a381 & +-cc & [6\_1]\\
12a395 & res & [8\_8]\\
12a396 & +-cc & [8\_20]\\
12a399 & res & [8\_8]\\
12a400 & cc & [9\_41]\\
12a413 & +-cc & [8\_20]\\
12a423 & res & [8\_8]\\
12a424 & +-cc & [3\_1\#-(3\_1)]\\
12a434 & +-cc & [8\_20]\\
12a436 & +-cc & [3\_1\#-(3\_1)]\\
12a438 & +-cc & [6\_1]\\
12a448 & +-cc & [6\_1]\\
12a449 & +-cc & [6\_1]\\
12a454 & +-cc & [6\_1]\\
12a459 & cc & [9\_46]\\
12a462 & +-cc & [3\_1\#-(3\_1)]\\
12a463 & +-cc & [6\_1]\\
12a465 & +-cc & [6\_1]\\
12a468 & +-cc & [8\_20]\\
12a481 & cc & [11n49]\\
12a482 & +-cc & [6\_1]\\
12a489 & +-cc & [6\_1]\\
12a493 & +-cc & [6\_1]\\
12a494 & +-cc & [6\_1]\\
12a496 & cc & [10\_129, 8\_8]\\
12a503 & cc & [10\_75]\\
12a505 & res & [8\_9, 4\_1\#4\_1]\\
12a544 & +-cc & [3\_1\#-(3\_1)]\\
12a545 & cc & [8\_8]\\
12a549 & +-cc & [6\_1]\\
12a554 & +-cc & [6\_1]\\
12a564 & res & [6\_1]\\
12a580 & cc & [10\_12]\\
12a581 & res & [3\_1\#-(3\_1)]\\
12a582 & res & [3\_1\#-(3\_1)]\\
12a597 & +-cc & [6\_1]\\
12a598 & +-cc & [8\_20]\\
12a601 & res & [6\_1]\\
12a609 & res & [8\_8]\\
12a621 & +-cc & [8\_20]\\
12a634 & res & [3\_1\#-(3\_1)]\\
12a639 & cc & [10\_87]\\
12a642 & +-cc & [8\_20]\\
12a643 & +-cc & [10\_129, 8\_8]\\
12a644 & +-cc & [10\_129, 8\_8]\\
12a649 & +-cc & [10\_129, 8\_8]\\
12a665 & +-cc & [6\_1]\\
12a668 & +-cc & [3\_1\#-(3\_1)]\\
12a669 & +-cc & [10\_129, 8\_8]\\
12a677 & cc & [8\_8]\\
12a680 & cc & [10\_87]\\
12a684 & +-cc & [6\_1]\\
12a687 & +-cc & [6\_1]\\
12a689 & +-cc & [6\_1]\\
12a690 & +-cc & [6\_1]\\
12a691 & +-cc & [6\_1]\\
12a692 & cc &  [3\_1\#-(3\_1)\#3\_1]\\
12a693 & +-cc & [8\_6]\\
12a704 & +-cc & [10\_129, 8\_8]\\
12a706 & +-cc & [3\_1\#-(3\_1)]\\
12a725 & res & [6\_2]\\
12a730 & +-cc & [7\_4]\\
12a735 & res & [6\_1]\\
12a741 & +-cc & [7\_4]\\
12a749 & res & [3\_1\#-(3\_1)]\\
12a750 & res & [6\_1]\\
12a752 & +-cc & [6\_1]\\
12a757 & cc & [8\_20]\\
12a767 & res & [6\_1]\\
12a769 & +-cc & [6\_1]\\
12a771 & res & [8\_8]\\
12a783 & +-cc & [6\_1]\\
12a784 & +-cc & [3\_1\#-(3\_1)]\\
12a789 & +-cc & [3\_1\#-(3\_1)]\\
12a791 & +-cc & [10\_129, 8\_8]\\
12a812 & +-cc & [7\_4]\\
12a815 & cc & [5\_1\#-(5\_1)]\\
12a816 & +-cc & [10\_129, 8\_8]\\
12a818 & +-cc & [3\_1\#-(3\_1)]\\
12a824 & cc & \footnotesize [10\_48, 5\_2\#-(5\_2)]\\
12a825 & +-cc & [8\_20]\\
12a826 & +-cc & [10\_129, 8\_8]\\
12a827 & +-cc & [3\_1\#-(3\_1)]\\
12a833 & cc & [10\_87]\\
12a835 & cc & \footnotesize [10\_48, 5\_2\#-(5\_2)]\\
12a841 & res & [6\_2]\\
12a842 & +-cc & [6\_1]\\
12a845 & +-cc & [3\_1\#-(3\_1)]\\
12a852 & cc & [3\_1\#-(3\_1)]\\
12a853 & cc & [3\_1\#-(3\_1)]\\
12a862 & res & [8\_8]\\
12a870 & cc & [8\_20]\\
12a871 & cc & [8\_20]\\
12a873 & +-cc & [8\_20]\\
12a878 & +-cc & [8\_20]\\
12a886 & res & [8\_9, 4\_1\#4\_1]\\
12a895 & cc & [10\_87]\\
12a896 & +-cc & [3\_1\#-(3\_1)]\\
12a898 & +-cc & [8\_20]\\
12a899 & +-cc & [10\_129, 8\_8]\\
12a901 & +-cc & [8\_20]\\
12a911 & cc & [10\_22, 10\_35]\\
12a912 & +-cc & [6\_1]\\
12a914 & res & [8\_8]\\
12a916 & cc & [3\_1\#-(3\_1)]\\
12a921 & res & [8\_9, 4\_1\#4\_1]\\
12a939 & res & [10\_3]\\
12a940 & +-cc & [6\_1]\\
12a941 & res & [10\_22, 10\_35]\\
12a942 & +-cc & \tiny[10\_137, 10\_155, 11n37]\\
12a957 & res & [8\_8]\\
12a967 & +-cc & [7\_4]\\
12a971 & +-cc & [6\_1]\\
12a981 & +-cc & [8\_20]\\
12a983 & +-cc & [7\_4]\\
12a988 & +-cc & [7\_4]\\
12a989 & res & [8\_8]\\
12a999 & +-cc & [8\_20]\\
12a1000 & +-cc & [8\_20]\\
12a1012 & +-cc & [8\_9, 4\_1\#4\_1]\\
12a1014 & +-cc & [6\_1]\\
12a1016 & +-cc & [6\_1]\\
12a1025 & res & [10\_22, 10\_35]\\
12a1028 & +-cc & [8\_20]\\
12a1039 & +-cc & [6\_1]\\
12a1040 & +-cc & [6\_1]\\
12a1050 & +-cc & [6\_1]\\
12a1061 & res & [9\_27]\\
12a1066 & cc & [10\_22, 10\_35]\\
12a1085 & +-cc & [6\_1]\\
12a1095 & +-cc & [6\_1]\\
12a1103 & res & [8\_8]\\
12a1109 & res & [9\_27]\\
12a1110 & res & [9\_41]\\
12a1115 & +-cc & [7\_4]\\
12a1116 & +-cc & [7\_4]\\
12a1124 & +-cc & [8\_9, 4\_1\#4\_1]\\
12a1127 & res & [6\_1]\\
12a1138 & +-cc & [6\_1]\\
12a1145 & cc & [10\_22, 10\_35]\\
12a1147 & cc & [10\_22, 10\_35]\\
12a1148 & +-cc & [6\_1]\\
12a1149 & +-cc & [6\_1]\\
12a1150 & +-cc & [6\_1]\\
12a1151 & +-cc & [6\_1]\\
12a1160 & res & [6\_1]\\
12a1163 & cc & [10\_3]\\
12a1165 & cc & [10\_3]\\
12a1171 & +-cc & [6\_1]\\
12a1174 & +-cc & [8\_20]\\
12a1175 & res & [9\_27]\\
12a1179 & +-cc & [6\_1]\\
12a1194 & cc & [10\_129, 8\_8]\\
12a1200 & +-cc & [6\_1]\\
12a1201 & +-cc & [8\_20]\\
12a1205 & +-cc & [10\_129, 8\_8]\\
12a1226 & +-cc & [8\_20]\\
12a1254 & +-cc & [8\_9, 4\_1\#4\_1]\\
12a1256 & +-cc & [6\_1]\\
12a1259 & +-cc & [6\_1]\\
12a1275 & +-cc & [6\_1]\\
12a1278 & +-cc & [6\_2]\\
12a1279 & +-cc & [6\_1]\\
12a1281 & res & [3\_1\#-(3\_1)]\\
12a1282 & cc & [10\_3]\\
12a1284 & +-cc & [8\_9, 4\_1\#4\_1]\\
12a1285 & +-cc & [8\_9, 4\_1\#4\_1]\\
12a1286 & +-cc & [8\_4]\\
12a1288 & +-cc & [8\_9, 4\_1\#4\_1]\\
12n47 & +-cc & [6\_1]\\
12n60 & +-cc & [3\_1\#-(3\_1)]\\
12n61 & +-cc & [3\_1\#-(3\_1)]\\
12n75 & +-cc & [3\_1\#-(3\_1)]\\
12n80 & cc & [8\_20]\\
12n84 & cc & [3\_1\#-(3\_1)]\\
12n92 & +-cc & [3\_1\#-(3\_1)]\\
12n101 & +-cc & [3\_1\#-(3\_1)]\\
12n109 & cc & [8\_20]\\
12n113 & res & [3\_1]\\
12n115 & cc & [10\_153]\\
12n116 & res & [U]\\
12n118 & res & [U]\\
12n137 & +-cc & [3\_1\#-(3\_1)]\\
12n140 & res & [8\_20]\\
12n147 & cc & [8\_8]\\
12n157 & res & [U]\\
12n159 & res & [U]\\
12n167 & +-cc & [10\_129, 8\_8]\\
12n171 & res & [U]\\
12n176 & res & [U]\\
12n190 & res & [8\_21]\\
12n192 & cc & [10\_153]\\
12n193 & res & [U]\\
12n197 & res & [8\_8]\\
12n200 & res & [6\_1]\\
12n202 & res & [8\_8]\\
12n204 & +-cc & [7\_4]\\
12n206 & +-cc & [6\_1]\\
12n208 & res & [U]\\
12n211 & +-cc & [6\_1]\\
12n212 & res & [U]\\
12n216 & cc & [8\_8]\\
12n219 & +-cc & [3\_1\#-(3\_1)]\\
12n227 & +-cc & [6\_1]\\
12n233 & res & [3\_1]\\
12n236 & res & [U]\\
12n247 & res & [U]\\
12n248 & res & [U]\\
12n253 & res & [U]\\
12n258 & +-cc & [6\_1]\\
12n260 & res & [U]\\
12n267 & cc & [8\_20]\\
12n270 & res & [U]\\
12n291 & cc & [3\_1\#-(3\_1)]\\
12n304 & res & [3\_1\#-(3\_1)]\\
12n307 & res & [6\_1]\\
12n324 & +-cc & [6\_1]\\
12n334 & cc & [6\_1]\\
12n345 & cc & [8\_20]\\
12n351 & +-cc & [6\_1]\\
12n359 & +-cc & [6\_1]\\
12n376 & cc & [8\_9, 4\_1\#4\_1]\\
12n379 & cc & [8\_20]\\
12n383 & res & [U]\\
12n388 & cc & [6\_1]\\
12n391 & res & [8\_8]\\
12n396 & +-cc & [6\_1]\\
12n409 & res & [U]\\
12n410 & res & [8\_8]\\
12n411 & res & [6\_1]\\
12n439 & cc & [3\_1\#-(3\_1)]\\
12n441 & res & [5\_2]\\
12n442 & +-cc & [6\_1]\\
12n443 & cc & [3\_1\#-(3\_1)]\\
12n451 & res & [U]\\
12n454 & res & [U]\\
12n456 & res & [U]\\
12n460 & +-cc & [6\_1]\\
12n469 & res & [U]\\
12n475 & res & [U]\\
12n480 & +-cc & [6\_1]\\
12n489 & res & [8\_8]\\
12n495 & +-cc & [8\_20]\\
12n496 & cc & [7\_4]\\
12n500 & res & [U]\\
12n514 & res & [U]\\
12n519 & res & [6\_1]\\
12n520 & res & [U]\\
12n522 & res & [U]\\
12n524 & +-cc & [6\_1]\\
12n525 & +-cc & [6\_1]\\
12n531 & +-cc & [3\_1\#-(3\_1)]\\
12n532 & +-cc & [6\_1]\\
12n537 & +-cc & [6\_1]\\
12n543 & res & [U]\\
12n554 & res & [6\_1]\\
12n564 & res & [U]\\
12n569 & res & [3\_1\#-(3\_1)]\\
12n577 & cc & [10\_140]\\
12n583 & cc & [6\_1]\\
12n596 & res & [3\_1\#-(3\_1)]\\
12n601 & res & [3\_1\#-(3\_1)]\\
12n606 & res & [U]\\
12n608 & res & [6\_1]\\
12n621 & res & [U]\\
12n626 & res & [7\_2]\\
12n630 & +-cc & [6\_1]\\
12n631 & +-cc & [6\_1]\\
12n672 & res & [10\_129, 8\_8]\\
12n673 & res & [U]\\
12n675 & cc & [3\_1\#-(3\_1)]\\
12n678 & +-cc & [8\_20]\\
12n681 & +-cc & [3\_1\#-(3\_1)]\\
12n685 & res & [U]\\
12n698 & res & [5\_2]\\
12n699 & res & [U]\\
12n700 & res & [3\_1]\\
12n701 & res & [U]\\
12n707 & res & [3\_1]\\
12n717 & res & [U]\\
12n726 & res & [U]\\
12n730 & res & [U]\\
12n734 & res & [3\_1]\\
12n735 & res & [U]\\
12n737 & res & [6\_1]\\
12n742 & res & [U]\\
12n759 & res & [6\_1]\\
12n769 & res & [U]\\
12n777 & res & [6\_1]\\
12n783 & res & [6\_1]\\
12n794 & +-cc & [6\_1]\\
12n796 & res & [5\_2]\\
12n797 & res & [U]\\
12n804 & res & [8\_9, 4\_1\#4\_1]\\
12n805 & +-cc & [6\_1]\\
12n808 & cc & [8\_20]\\
12n809 & cc & [11n116]\\
12n811 & res & [8\_20]\\
12n813 & +-cc & [6\_1]\\
12n814 & res & [U]\\
12n815 & +-cc & [6\_1]\\
12n818 & +-cc & [6\_1]\\
12n822 & res & [U]\\
12n824 & +-cc & [6\_1]\\
12n829 & cc & [8\_20]\\
12n833 & +-cc & [8\_20]\\
12n844 & +-cc & [6\_1]\\
12n846 & +-cc & [6\_1]\\
12n854 & res & [8\_8]\\
12n855 & +-cc & [6\_1]\\
12n856 & +-cc & [6\_1]\\
12n859 & +-cc & [6\_1]\\
12n861 & res & [U]\\
12n862 & res & [U]\\
12n863 & res & [5\_2]\\
12n867 & cc & [7\_4]\\
12n869 & cc & [8\_20]\\
12n873 & +-cc & [6\_1]\\
12n875 & +-cc & [9\_46]\\
\hline
\end{supertabular}

\onecolumn
\section*{Appendix B} \label{app:taylor}
This appendix contains Seifert matrices (from KnotInfo \cite{Knotinfo}) and bases of isotropic subspaces
(computed with PARI/GP \cite{pari}) of the five knots referred to at the end of \cref{sec:taylor}.
Each of them has Taylor invariant equal to $1$.

\tablehead{\hline Knot & Seifert matrix & Basis of an isotropic subgroup  \\ \hline}
\tabletail{\hline}
\newcommand{\paddedpmatrix}[1]{\raisebox{0pt}[\dimexpr+\height+1ex][\dimexpr+\depth+1ex]{$\begin{pmatrix}#1\end{pmatrix}$}}
\begin{longtable}{|l | c | c | }
\hline Knot & Seifert matrix & Basis of an isotropic subgroup \\ \hline
\endhead
\hline
\endfoot
12a244 &
\paddedpmatrix{
-1	&	0	&	0	&	0	&	0	&	0\\
0	&	2	&	0	&	0	&	1	&	1\\
-1	&	0	&	-1	&	0	&	0	&	0\\
-1	&	-1	&	-1	&	1	&	-1	&	-1\\
0	&	1	&	0	&	0	&	2	&	1\\
0	&	1	&	0	&	0	&	0	&	2
} &
\paddedpmatrix{
 2  \\
 4  \\
 2  \\
 3  \\
-1  \\
-1  \\
},
\paddedpmatrix{
  0 \\
 -4 \\
 -4 \\
 -2 \\
  0 \\
  1 \\
}
\\\hline

12a810 &
\paddedpmatrix{
1	&	0	&	0	&	0	&	0	&	0\\
0	&	-1	&	0	&	1	&	-1	&	1\\
-1	&	0	&	-1	&	-1	&	1	&	-1\\
1	&	0	&	0	&	1	&	0	&	0\\
0	&	0	&	0	&	1	&	-2	&	2\\
0	&	0	&	0	&	-1	&	1	&	-3
} &
\paddedpmatrix{
-1  \\
 2  \\
 3  \\
 5  \\
 3  \\
 2  \\
},
\paddedpmatrix{
  0 \\
  7 \\
  6 \\
 11 \\
  6 \\
  5 \\
}
\\\hline

12a905 &
\paddedpmatrix{
-1	&	0	&	0	&	0	&	0	&	-1\\
-1	&	1	&	1	&	1	&	1	&	-1\\
-1	&	0	&	1	&	1	&	1	&	0\\
1	&	0	&	0	&	-2	&	0	&	1\\
0	&	0	&	0	&	1	&	1	&	0\\
0	&	0	&	0	&	0	&	0	&	-2
} &
\paddedpmatrix{
 1  \\
 1  \\
-1  \\
 1  \\
 1  \\
-1  \\
},
\paddedpmatrix{
 0 \\
-1 \\
 0 \\
 0 \\
 0 \\
 1 \\
}
\\\hline

12n555 &
\paddedpmatrix{
1&0&0&0&0&0\\
0&1&0&0&0&0\\
0&0&1&0&0&0\\
1&0&0&1&1&1\\
-1&-1&0&0&-1&-1\\
-1&0&-1&0&0&-1
} &
\paddedpmatrix{
 1 \\
-1 \\
 0 \\
-1 \\
 1 \\
 0
},
\paddedpmatrix{
0 \\
1 \\
0 \\
0 \\
0 \\
1
}
\\\hline

12n642 &
\paddedpmatrix{
1&0&0&0\\
1&1&0&-1\\
1&-1&1&-1\\
1&0&0&1
} &
\paddedpmatrix{
1 \\ -1 \\ -1 \\ 0
}
\end{longtable}

\section*{Appendix C} \label{app:galg}
For each knot $K$ in the second table of \cref{sec:galg}, we give a Seifert matrix (from KnotInfo \cite{Knotinfo}) and a basis of an Alexander-trivial subgroup of maximal rank
(computed with PARI/GP \cite{pari} using a randomized algorithm as in \cite{BaaderFellerLewarkLiechti_15}).
The basis is chosen such that the matrix of the restriction of the Seifert form to the subgroup with respect to the basis has the following form:
\[
\begin{pmatrix}
0      & 1      & 0 & \cdots & 0 \\
0      & *      & * & \cdots & * \\
\vdots & \vdots &   & \raisebox{0pt}[0pt][0pt]{\makebox[0pt][c]{\fbox{\raisebox{0ex}[3ex][2ex]{\hspace{1.2em}M\hspace{1.2em}}}}}              \\
0      & *      &
\end{pmatrix},
\]
where $M$ is a quadratic matrix of dimension two less, of the same form.
\tablehead{\hline Knot & Seifert matrix & Basis of an Alexander-trivial subgroup  \\ \hline}
\tabletail{\hline}
\begin{longtable}{|l | c | c | }
\hline Knot & Seifert matrix & Basis of an Alexander-trivial subgroup \\ \hline
\endhead
\hline
\endfoot
11n80      & \paddedpmatrix{ 1&0&0&0&0&0\\ 1&0&-1&0&-1&0\\ 1&-1&-2&0&-1&0\\ 0&1&0&1&1&0\\ 1&0&-1&0&0&0\\ 0&1&0&1&1&1 } &
\paddedpmatrix{0\\1\\0\\0\\0\\0}, \paddedpmatrix{1\\0\\0\\0\\0\\0}, \paddedpmatrix{0\\0\\1\\1\\-1\\0}, \paddedpmatrix{0\\0\\0\\1\\0\\-1}\\\hline
12a187     & \paddedpmatrix{ -1&-1&0&0&0&0\\ 0&1&0&0&0&0\\ 0&0&1&0&0&0\\ -1&-1&1&-1&0&0\\ 0&0&0&-1&1&0\\ 1&0&0&1&-1&2} &
\paddedpmatrix{1\\0\\1\\0\\0\\0}, \paddedpmatrix{0\\-1\\0\\0\\0\\0}, \paddedpmatrix{0\\-1\\-1\\2\\2\\1}, \paddedpmatrix{0\\0\\0\\0\\1\\0}\\\hline
12a230     & \paddedpmatrix{ 1&0&0&0&0&0\\ -1&-1&-1&0&0&0\\ -1&0&-1&0&-1&1\\ 0&0&0&2&1&1\\ 0&0&0&0&1&0\\ 0&0&0&1&0&2} &
\paddedpmatrix{0\\1\\0\\0\\1\\0}, \paddedpmatrix{-1\\0\\0\\0\\0\\0}, \paddedpmatrix{0\\0\\2\\0\\2\\1}, \paddedpmatrix{-1\\0\\1\\1\\0\\0}\\\hline
12a317     & \paddedpmatrix{ -1&0&0&0&0&0\\ 0&2&0&1&1&0\\ -1&0&-1&0&0&0\\ -1&0&-1&1&0&0\\ 0&1&0&1&2&1\\ 1&1&1&0&1&2} &
\paddedpmatrix{-2\\0\\0\\0\\-2\\1}, \paddedpmatrix{0\\0\\-5\\0\\-2\\-2}, \paddedpmatrix{-1\\1\\2\\-1\\0\\-1}, \paddedpmatrix{-1\\0\\1\\-1\\0\\-1}\\\hline
12a450     & \paddedpmatrix{ 2&0&0&0&1&1\\ 0&-1&-1&0&0&0\\ 0&0&1&0&0&0\\ 0&0&0&1&0&0\\ 0&-1&0&0&1&0\\ 0&-1&-1&1&0&-1} &
\paddedpmatrix{0\\-1\\1\\0\\0\\1}, \paddedpmatrix{0\\1\\0\\1\\0\\0}, \paddedpmatrix{-1\\1\\0\\0\\1\\0}, \paddedpmatrix{-2\\1\\0\\1\\3\\1}\\\hline
12a542     & \paddedpmatrix{ 1&0&0&0\\ -1&-1&-1&-1\\ -1&0&-3&-1\\ -1&0&-2&-3} &
\paddedpmatrix{3\\-1\\-1\\2}, \paddedpmatrix{15\\0\\-16\\11}\\\hline
12a570     & \paddedpmatrix{ 2&0&0&0&0&0\\ 0&-1&0&-1&0&-1\\ 0&0&1&0&0&0\\ 0&0&0&1&0&0\\ -1&-1&0&0&1&0\\ 0&0&1&-1&0&-1} &
\paddedpmatrix{0\\-1\\-1\\2\\1\\2}, \paddedpmatrix{5\\5\\0\\0\\5\\-1}, \paddedpmatrix{0\\0\\1\\-1\\-1\\-1}, \paddedpmatrix{0\\0\\1\\0\\-2\\-1}\\\hline
12a908     & \paddedpmatrix{ -1&0&-1&0&0&-1\\ -1&1&-1&1&1&0\\ 0&0&-2&0&0&-2\\ 1&0&1&-2&0&1\\ 0&0&0&1&1&0\\ 0&0&-1&0&0&-2} &
\paddedpmatrix{0\\1\\-1\\0\\0\\0}, \paddedpmatrix{0\\1\\-4\\-4\\4\\2}, \paddedpmatrix{3\\0\\-3\\-2\\6\\1}, \paddedpmatrix{-14\\6\\7\\7\\-24\\-5}\\\hline
12a1118    & \paddedpmatrix{ -1&0&0&0&0&0\\ 0&1&0&0&0&0\\ -1&1&1&0&0&1\\ 0&-1&-1&3&-1&-1\\ -1&0&0&0&-1&0\\ -1&1&0&0&0&1} &
\paddedpmatrix{2\\1\\0\\1\\-2\\1}, \paddedpmatrix{3\\0\\2\\2\\0\\3}, \paddedpmatrix{0\\0\\1\\0\\1\\0}, \paddedpmatrix{2\\1\\2\\1\\0\\2}\\\hline
12a1185    & \paddedpmatrix{ -1&-1&-1&0&0&-1\\ 0&1&1&0&0&0\\ 0&0&1&0&0&0\\ 0&1&0&-2&1&0\\ 0&-1&0&0&-1&0\\ 0&0&0&-1&1&-2} &
\paddedpmatrix{1\\-1\\2\\0\\0\\-1}, \paddedpmatrix{0\\-2\\-2\\-2\\2\\1}, \paddedpmatrix{0\\1\\2\\2\\-1\\-1}, \paddedpmatrix{1\\0\\1\\0\\-1\\0}\\\hline
12a1189    & \paddedpmatrix{ 1&0&0&0&0&0\\ 0&-1&0&1&-1&0\\ -1&0&-1&-1&0&1\\ 1&0&0&1&0&0\\ -1&0&0&0&-2&1\\ 1&0&0&1&0&-2} &
\paddedpmatrix{1\\0\\-1\\-1\\0\\0}, \paddedpmatrix{0\\0\\0\\0\\0\\-1}, \paddedpmatrix{-1\\1\\1\\-1\\-1\\0}, \paddedpmatrix{-1\\0\\0\\0\\-1\\-1}\\\hline
12a1208    & \paddedpmatrix{ -1&-1&-1&1&-1&0\\ 0&1&1&0&0&0\\ 0&0&1&0&0&0\\ 0&1&0&-2&1&0\\ 0&0&0&0&-2&1\\ 0&-1&0&0&0&-1} &
\paddedpmatrix{-1\\1\\-1\\0\\0\\-1}, \paddedpmatrix{0\\0\\0\\0\\0\\1}, \paddedpmatrix{0\\-1\\-2\\-1\\1\\3}, \paddedpmatrix{0\\1\\1\\0\\-1\\-3}\\\hline
12n52      & \paddedpmatrix{ 1&0&0&0&0&0\\ -1&-1&0&0&1&-1\\ 0&0&1&0&0&0\\ 0&0&0&1&0&0\\ 1&0&1&0&1&1\\ -1&0&-1&-1&0&-1} &
\paddedpmatrix{1\\1\\-1\\-1\\0\\-1}, \paddedpmatrix{0\\-1\\0\\1\\0\\0}, \paddedpmatrix{0\\1\\0\\0\\1\\1}, \paddedpmatrix{0\\0\\1\\-1\\0\\0}\\\hline
12n63      & \multicolumn{2}{l|}{\paddedpmatrix{ -1&0&-1&-1&-1&-1&0&0\\ 0&-1&0&0&0&0&0&0\\ 0&0&1&0&0&0&0&0\\ 0&1&-1&0&-1&-1&0&0\\ 0&1&-1&0&0&0&0&0\\ 0&1&-1&0&-1&0&0&0\\ 0&0&0&-1&-1&-1&1&1\\ 0&0&0&0&0&0&0&1}} \\\cline{2-3}
 & \multicolumn{2}{r|}{%
\paddedpmatrix{0\\0\\0\\1\\0\\0\\0\\0}, \paddedpmatrix{0\\1\\0\\0\\0\\0\\0\\0}, \paddedpmatrix{0\\1\\0\\0\\0\\1\\0\\0}, \paddedpmatrix{0\\-1\\0\\0\\-1\\0\\0\\0}, \paddedpmatrix{-1\\0\\0\\0\\0\\0\\1\\0}, \paddedpmatrix{0\\0\\0\\0\\0\\0\\0\\1}}\\\hline
12n225     & \paddedpmatrix{ -1&-1&-1&-1&0&-1\\ 0&-2&-1&-2&0&-1\\ 0&-1&0&-1&0&-1\\ 0&-1&-1&-2&0&-1\\ 0&0&0&0&-1&-1\\ 0&-1&-1&-1&0&0} &
\paddedpmatrix{0\\0\\1\\0\\0\\0}, \paddedpmatrix{1\\0\\0\\0\\0\\-1}, \paddedpmatrix{0\\0\\0\\1\\0\\-1}, \paddedpmatrix{0\\1\\1\\0\\0\\-1}\\\hline
12n276     & \paddedpmatrix{ 1&0&0&0&0&0\\ 0&1&0&0&0&0\\ 1&0&1&0&0&0\\ 1&0&0&1&0&0\\ 1&1&-1&0&1&-1\\ 1&0&1&0&0&1} &
\paddedpmatrix{-2\\0\\1\\1\\1\\0}, \paddedpmatrix{2\\0\\0\\-1\\-1\\-1}\\\hline
12n558     & \paddedpmatrix{ -1&0&-1&0&0&0\\ 0&-1&1&1&-1&1\\ 0&0&-1&1&-1&1\\ 0&0&0&1&0&1\\ 0&0&0&0&-1&0\\ 0&0&0&0&0&1} &
\paddedpmatrix{1\\0\\0\\1\\0\\0}, \paddedpmatrix{0\\0\\0\\0\\0\\1}, \paddedpmatrix{0\\1\\0\\-1\\0\\1}, \paddedpmatrix{0\\1\\0\\-1\\-1\\1}\\\hline
12n665     & \multicolumn{2}{l|}{\paddedpmatrix{ 1&0&0&0&0&0&0&0\\ 0&-1&0&1&-1&0&-1&0\\ -1&0&-1&-1&1&0&1&0\\ 1&0&0&1&0&0&0&0\\ 0&0&0&1&0&0&-1&0\\ 0&0&0&0&1&1&1&0\\ 0&0&0&1&0&0&0&0\\ 0&0&0&0&1&1&1&1}} \\ \cline{2-3}
 & \multicolumn{2}{r|}{\paddedpmatrix{0\\1\\0\\0\\0\\1\\0\\0}, \paddedpmatrix{0\\0\\-1\\0\\-2\\1\\0\\-1}, \paddedpmatrix{0\\0\\0\\0\\0\\0\\1\\0}, \paddedpmatrix{0\\0\\1\\1\\1\\-1\\0\\1}, \paddedpmatrix{0\\0\\1\\0\\1\\0\\0\\0}, \paddedpmatrix{1\\0\\0\\0\\0\\0\\1\\0}}\\\hline
12n886     & \paddedpmatrix{ 1&0&-1&0&0&0\\ 0&-1&0&-1&0&0\\ 0&0&-1&1&0&0\\ 0&0&0&-2&0&0\\ 1&1&-1&1&1&0\\ -1&0&1&-1&-1&-1} &
\paddedpmatrix{1\\0\\1\\0\\1\\0}, \paddedpmatrix{0\\1\\0\\0\\0\\0}, \paddedpmatrix{0\\1\\1\\0\\2\\1}, \paddedpmatrix{0\\2\\2\\1\\2\\0}\\
\end{longtable}
\end{small}
\end{appendix}
\end{document}